\documentclass{amsart}

\usepackage{amssymb}
\usepackage{amscd}
\usepackage{graphicx}
\usepackage{psfrag}

\newtheorem{theorem}{Th\'eor\`eme}[section]
\newtheorem*{theo}{Th\'eor\`eme}

\newtheorem{lemma}[theorem]{Lemme}
\theoremstyle{definition}

\newtheorem{proposition}[theorem]{Proposition}

\theoremstyle{remark}
\newtheorem{remark}[theorem]{Remarque}

\numberwithin{equation}{section}

\begin{document}

\title[]{Divergence et parall\'elisme des rayons d'\'etirement cylindriques} 

\author{Guillaume Th\'eret}
\address{Max Planck Institut f\"ur Mathematik\\
 Vivatsgasse 7 \\ 
 53100 Bonn, Deutschland}
\curraddr{}
\email{theret@mpim-bonn.mpg.de}
\thanks{}


\date{}

\dedicatory{}

\begin{abstract}
A cylindrical stretch line is a stretch line, in the sense of Thurston, 
whose horocyclic lamination is a weighted multicurve.
In this paper, we show that two correctly parameterized cylindrical lines are parallel
if and only if these lines converge towards the same point in Thurston's boundary 
of Teichm\"uller space.
\end{abstract}

\maketitle

\section{Introduction}

Fixons une surface ferm\'ee orientable $\Sigma$ de genre fini sup\'erieur ou \'egal \`a deux.
Notons $\mathcal{T}(\Sigma)$ l'espace de Teichm\"uller associ\'e \`a la surface $\Sigma$.
Dans cet article, nous munissons $\mathcal{T}(\Sigma)$ de la m\'etrique asym\'etrique de Thurston $d_{\mathcal{T}}$.
Cette m\'etrique est donn\'ee par la formule
$$
d_{\mathcal{T}}(g,h)=\log\sup_{\alpha\in\mathcal{ML}(\Sigma)}\frac{\ell_{h}(\alpha)}{\ell_{g}(\alpha)},
$$
o\`u $\mathcal{ML}(\Sigma)$ d\'esigne l'espace des laminations g\'eod\'esiques mesur\'ees et 
$\ell_{h}(\alpha)$ d\'esigne la longueur de la lamination g\'eod\'esique mesur\'ee $\alpha$ pour 
la classe d'isotopie $h\in\mathcal{T}(\Sigma)$ de m\'etriques hyperboliques.
L'adjectif "asym\'etrique" insiste sur le fait qu'il existe des points $g,h$ de $\mathcal{T}(\Sigma)$ pour lesquels
$d_{\mathcal{T}}(g,h)\neq d_{\mathcal{T}}(h,g)$.

Une ligne d'\'etirement de Thurston est une g\'eod\'esique orient\'ee $t\mapsto h_{t}$ de l'espace de Teichm\"uller
pour la m\'etrique de Thurston, c'est-\`a-dire que l'on a, pour tout $s\leq t$, 
$$
d_{\mathcal{T}}(h_{s},h_{t})=t-s.
$$
Dans cet article, toutes les param\'etrisations seront positives (i.e., pr\'eserveront l'orientation) et par la longueur d'arc.

Une ligne d'\'etirement est d\'etermin\'ee par deux laminations g\'eod\'esiques,
l'une \'etant \emph{compl\`ete}, appel\'ee le \textbf{support} de la ligne d'\'etirement 
et g\'en\'eralement not\'ee $\mu$, et l'autre \'etant 
la classe projective d'une lamination mesur\'ee totalement transverse \`a $\mu$,
appel\'ee la \textbf{direction} de la ligne d'\'etirement. 
La \textbf{souche} d'une ligne d'\'etirement est la sous-lamination du support de la ligne
qui poss\`ede une mesure transverse de support total et qui est maximale au sens de l'inclusion. 

Une \textbf{multicourbe} est une r\'eunion de courbes simples ferm\'ees disjointes.
La donn\'ee d'une mesure transverse sur une multicourbe est \'equivalente \`a la donn\'ee 
d'une pond\'eration positive des composantes de cette multicourbe.

Dans cet article, nous imposerons aux supports des lignes d'\'etirement d'\^etre des laminations
\textbf{r\'ecurrentes par cha\^ines}, c'est-\`a-dire limites de multicourbes pour la topologie de Hausdorff.

On dira d'une ligne d'\'etirement qu'elle est \textbf{cylindrique} 
si sa direction est la classe projective d'une multicourbe pond\'er\'ee. 

Nous dirons de deux lignes d'\'etirement $t\mapsto g_{t}$, $t\mapsto h_{t}$,
qu'elles \textbf{divergent} si les distances $d_{\mathcal{T}}(g_{t},h_{t})$ et $d_{\mathcal{T}}(h_{t},g_{t})$
tendent vers l'infini lorsque $t$ tend vers $+\infty$.
Enfin, nous dirons de ces lignes qu'elles sont \textbf{parall\`eles} si, quitte \`a changer la param\'etrisation,
les distances pr\'ec\'edentes restent born\'ees pour $t$ assez grand (nous verrons un peu plus loin que 
la possibilit\'e de changement de param\'etrisation dans cette d\'efinition est importante, 
contrairement au cas des m\'etriques sym\'etriques).

Le but de ce papier est de d'\'etablir le r\'esultat suivant.

\begin{theo}
Deux lignes d'\'etirement cylindriques, de supports r\'ecurrents par cha\^ines, 
sont parall\`eles si et seulement si elles ont m\^eme direction.
\end{theo}

Nous d\'emontrerons en fait quelque chose d'un peu plus pr\'ecis, \`a savoir, que deux lignes d'\'etirement cylindriques
non parall\`eles divergent, quitte \`a les reparam\'etrer.\\

Notre r\'esultat est en contraste avec la situation o\`u l'espace de Teichm\"uller est muni de la m\'etrique de 
Teichm\"uller.
Dans ce cas, H. Masur \cite{masur} a montr\'e l'existence de lignes d'\'etirement de Teichm\"uller cylindriques (dites 
de Jenkins-Strebel dans ce contexte) dirig\'ees par des classes projectives distinctes mais n\'eanmoins parall\`eles.\\

Le fait qu'il faille \'eventuellement changer la param\'etrisation des lignes d'\'etirement vient du r\'esultat suivant,
qui sera d\'emontr\'e dans cet article.

\begin{theo}
Soit $t\mapsto h_{t}$ une ligne d'\'etirement cylindrique de support r\'ecurrent par cha\^ines.
Pour tout nombre strictement positif $c$, on a
$$
\forall t,\ d_{\mathcal{T}}(h_{t},h_{t+c})=c,\ \textrm{et}\ \lim_{t\to\infty}d_{\mathcal{T}}(h_{t+c},h_{t})=\infty.
$$
\end{theo}

En effet, ce r\'esultat montre que si l'on ne prend pas garde \`a reparam\'etrer la ligne, deux lignes d'\'etirement dont les images co\"incident
pourraient ne pas \^etre parall\`eles. 
Notez que deux lignes d'\'etirement divergentes ne sont pas parall\`eles.

\section{Pr\'eliminaires}

Nous rappelons ici tr\`es succintement la construction des lignes d'\'etirement de Thurston \cite{thu}.

Fixons une lamination g\'eod\'esique compl\`ete $\mu$ ainsi qu'une m\'etrique hyperbolique $h$ sur $\Sigma$.
Chaque composante de $\Sigma\setminus\mu$ est l'int\'erieur d'un triangle id\'eal.
On construit d'abord un feuilletage partiel de l'int\'erieur de chaque triangle id\'eal de $\Sigma\setminus\mu$.
Les feuilles de ce feuilletage partiel sont des arcs d'horocycles centr\'es aux sommets du triangle.
Le feuilletage est invariant par la sym\'etrie d'ordre 3 de chaque triangle.
Une r\'egion triangulaire bord\'ee par 3 arcs d'horocycles de longueur 1 reste non-feuillet\'ee.
On \'etend ensuite par continuit\'e ce feuilletage partiel d\'efini sur $\Sigma\setminus\mu$ en un feuilletage partiel sur $\Sigma$.
Il y a alors exactement une r\'egion non-feuillet\'ee par triangle id\'eal de $\mu$.
On munit ce feuilletage partiel d'une mesure transverse en d\'ecr\'etant que la mesure
d'un arc compact transverse est la longueur de sa projection le long des feuilles du feuilletage partiel 
sur une feuille de $\mu$.
On note ce feuilletage partiel mesur\'e $F_{\mu}(h)$ et on l'appelle le \textbf{feuilletage horocyclique} associ\'e \`a $\mu$ et $h$.
La classe, au sens des feuilletages mesur\'es, du feuilletage horocyclique est not\'ee de la m\^eme fa\c con.\\

W. Thurston a montr\'e que la construction pr\'ec\'edente,
$$
F_{\mu}\ :\ \mathcal{T}(\Sigma)\to\mathcal{MF}(\Sigma),\quad h\mapsto F_{\mu}(h),
$$
est un hom\'eomorphisme sur son image.
Cette image est constitu\'ee, dans le cas des surfaces ferm\'ees, des classes de feuilletages
mesur\'es transverses \`a $\mu$.

La \textbf{ligne d'\'etirement} de support $\mu$ et passant par $h\in\mathcal{T}(\Sigma)$ est la courbe 
$$
t\mapsto h_{t}=F_{\mu}^{-1}(e^{t}F_{\mu}(h)),
$$
o\`u $t$ varie dans $\mathbb{R}$ et o\`u la notation $e^{t}F_{\mu}(h)$ signifie que l'on a multipli\'e la mesure transverse
du feuilletage horocyclique $F_{\mu}(h)$ par $e^{t}$.
Notez que, par d\'efinition,
$$
h_{0}=h\ \textrm{et}\ F_{\mu}(h_{t})=e^{t}F_{\mu}(h).
$$
L'image, par l'application $F_{\mu}$, de la ligne d'\'etirement passant par $h$ et de support $\mu$ 
est l'ensemble $\{e^{t}F_{\mu}(h)\ :\ t\in\mathbb{R}\}$.
Tous les points d'une ligne d\'etirement d\'efinissent des feuilletages horocycliques qui appartiennent \`a la m\^eme classe projective,
qu'on appelle la \textbf{direction} de la ligne d'\'etirement.
Ainsi, une ligne d'\'etirement, comme sous-ensemble de $\mathcal{T}(\Sigma)$, est d\'etermin\'ee par son support et sa direction.

Dans la suite, nous parlerons \'egalement de la \textbf{lamination horocyclique} $\lambda_{\mu}(h)$ qui n'est rien d'autre que 
la lamination g\'eod\'esique mesur\'ee associ\'ee au feuilletage horocyclique $F_{\mu}(h)$.

\section{\'Etirements \'el\'ementaires cylindriques et longueur de la lamination horocyclique}

Le but de cette partie est d'\'etablir les r\'esultats pr\'esent\'es plus haut dans le cas particulier o\`u les 
supports des lignes d'\'etirements sont des laminations g\'eod\'esiques \'el\'ementaires.
Une lamination g\'eod\'esique est \textbf{\'el\'ementaire} si elle est compl\`ete, 
r\'ecurrente par cha\^ines et si sa souche est une multicourbe.
De mani\`ere \'equivalente, une lamination \'el\'ementaire est une lamination 
g\'eod\'esique compl\`ete compos\'ee d'un nombre fini de feuilles ferm\'ees 
et d'un nombre fini de feuilles infinies, isol\'ees, spiralant autour de chaque feuille ferm\'ee 
de telle mani\`ere que, pour un observateur situ\'e sur la feuille ferm\'ee $a$ et regardant les spirales, 
celles-ci tournent dans la m\^eme direction pour les deux c\^ot\'es de $a$.\\

Soit $\mu$ une lamination g\'eod\'esique compl\`ete \'el\'ementaire de souche $\gamma$.
Supposons que la structure hyperbolique $h\in\mathcal{T}(\Sigma)$ soit choisie de telle sorte que la 
lamination horocyclique $\lambda_{\mu}(h)$ -- que nous noterons desormais $\lambda$ pour faire court --
soit une multicourbe, ou, pour reprendre le langage de l'introduction, de telle sorte que la ligne d'\'etirement
de support $\mu$ et de direction la classe de $\lambda$ soit cylindrique.

Chaque composante $\lambda_{j}$, $j\in\{1,\cdots,M\}$, de $\lambda$ d\'efinit un cylindre $C_{j}$ qui est
la cl\^oture de la r\'eunion des feuilles du feuilletage horocyclique $F_{\mu}(h)$ qui sont isotopes \`a $\lambda_{j}$.
La g\'eod\'esique ferm\'ee simple $\lambda_{j}$ est appel\'ee le \textbf{c\oe ur} du cylindre $C_{j}$.
Chacun des deux bords du cylindre est une feuille ferm\'ee singuli\`ere, ce qui signifie
qu'elle contient un nombre fini de c\^ot\'es de r\'egions non-feuillet\'ees du feuilletage horocyclique.
Nous dirons que ces r\'egions non-feuillet\'ees sont \textbf{adjacentes} au cylindre. 
Par d\'efinition de la mesure transverse de $F_{\mu}(h)$, 
les feuilles de $\mu\cap C_{j}$ ont toutes m\^eme longueur, $w_{j}$.
Le nombre $w_{j}$, qu'on appellera la \textbf{largeur} du cylindre $C_{j}$, est aussi le poids
de la composante $\lambda_{j}$ vue comme lamination g\'eod\'esique mesur\'ee.

Nous allons donner un \'equivalent asymptotique de la longueur de la g\'eod\'esique $\lambda_{j}$ lorsqu'on \'etire 
le long de $\mu$ ind\'efiniment.

\begin{proposition}
\label{pr:equiv_elem}
Soit $t\mapsto h_{t}$ une ligne d'\'etirement \'el\'ementaire cylindrique de support $\mu$, passant par $h=h_{0}$.
La lamination horocyclique $\lambda_{\mu}(h)$ est la multicourbe $\lambda=\lambda_{1}\cup\cdots\cup\lambda_{M}$ 
avec les pond\'erations $w_{1},\dots,w_{M}$.
Lorsque $t$ tend vers $+\infty$, 
la longueur $\ell_{h_{t}}(\lambda_{j})$ de la courbe simple ferm\'ee $\lambda_{j}$ est \'equivalente \`a la quantit\'e
$$
2\sqrt{K_{j}}e^{-e^{t}w_{j}/2},
$$
o\`u $K_{j}\in\mathbb{N}$ est le nombre de r\'egions non-feuillet\'ees adjacentes \`a un bord du cylindre de c\oe ur $\lambda_{j}$
multipli\'e par le nombre de r\'egions non-feuillet\'ees adjacentes \`a l'autre bord. 
\end{proposition}

\subsection{Calculs dans un cylindre}

Soit $C$ un cylindre parmi les cylindres $C_{1},\cdots,C_{M}$.
Notons $w$ sa largeur.
La multicourbe $\gamma$ traverse le cylindre $C$ en un nombre fini de segments.
Au voisinage de chacun de ces segments, les feuilles isol\'ees de $\mu$ spiralent 
dans un sens qui est le m\^eme de chaque c\^ot\'e du segment. 

Une \textbf{bande} de $C$ d\'esignera une partie obtenue en
coupant le cylindre $C$ le long de deux feuilles de $\mu\cap C$, 
de telle sorte que toutes les feuilles isol\'ees de $\mu$ \`a l'int\'erieur de cette bande 
ont le m\^eme sommet id\'eal.

Ainsi, une bande est isom\'etrique \`a un rectangle du demi-plan hyperbolique, 
dont les bords horizontaux sont des horocycles centr\'es en l'infini et les bords verticaux des g\'eod\'esiques asymptotes
\`a l'infini.
Dans cette description, les feuilles de $\mu$ traversant $C$ correspondent \`a des g\'eod\'esiques verticales asymptotes \`a l'infini.
La restriction du feuilletage horocyclique correspond quant \`a lui au feuilletage du rectangle par des segments euclidiens horizontaux,
c'est-\`a-dire par des arcs d'horocycles centr\'es \`a l'infini.
Les longueurs de ces arcs d'horocycles sont strictement d\'ecroissantes \`a mesure que l'on se rapproche lin\'eairement du sommet id\'eal. 
L'\textbf{\'epaisseur} d'une bande est la longueur du plus grand arc d'horocycle.

Une bande est dite \textbf{minimale} s'il n'existe pas de bande strictement contenue en elle.
Une bande minimale est donc la partie d'un triangle id\'eal de $\mu$ comprise entre deux feuilles du feuilletage horocyclique.

Une bande est dite \textbf{maximale} s'il n'existe pas de bande la contenant strictement.  
Il existe une unique d\'ecomposition du cylindre $C$ en bandes maximales d'int\'erieurs disjoints.
On l'appelle la \textbf{d\'ecomposition} de $C$ \textbf{en bandes}.
Les arcs d'horocycles de deux bandes adjacentes ont des convexit\'es de signes oppos\'es (voir la figure \ref{bande1}).
Notez que la d\'ecomposition en bandes contient un nombre pair de bandes.

\begin{figure}[!hbp]
\centering
\psfrag{l}{\small $a$}
\psfrag{r}{\small $b$}
\includegraphics[width=.5\linewidth]{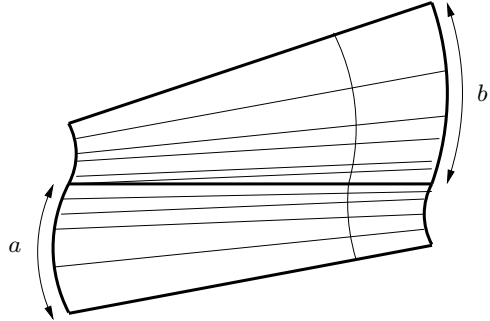}
\caption{\small{Ce dessin repr\'esente deux bandes adjacentes de la d\'ecomposition en bandes d'un cylindre.
La fronti\`ere commune dans ce dessin est une composante de $\gamma\cap C$ et est repr\'esent\'ee horizontalement, en trait \'epais. 
Une feuille int\'erieure du feuilletage horocyclique a \'et\'e dessin\'ee, en trait fin ; elle est obtenue en concat\'enant deux
arcs d'horocycles de convexit\'es de signes oppos\'es. Les \'epaisseurs des bandes sont not\'ees $a$ et $b$.}}
\label{bande1}
\end{figure}

Notons $B_{1},\cdots,B_{2N}$ la d\'ecomposition de $C$ en bandes, avec $B_{j}$ adjacente \`a $B_{j+1}$.
Notons  $a_{1},\cdots,a_{2N}$ les \'epaisseurs respectives de ces bandes.
Orientons les feuilles de $\mu\cap C$ de telle sorte que les arcs d'horocycles repr\'esentant les \'epaisseurs 
d'indices impairs se trouvent dans le bord gauche de $C$ et les arcs d'horocycles repr\'esentant les \'epaisseurs 
d'indices pairs se trouvent dans le bord droit de $C$.
Posons
$$
a_{p}=\sum_{k=1}^{N}a_{2k}\ \textrm{et}\ a_{i}=\sum_{k=1}^{N}a_{2k-1}.
$$
On param\`etre l'ensemble des feuilles du feuilletage horocyclique de $C$ par la distance, par rapport au bord droit de $C$, 
\`a laquelle elles intersectent n'importe quelle feuille de $\mu\cap C$.
On note $h^{*}(d)$ la longueur de la feuille ferm\'ee du feuilletage horocyclique situ\'ee \`a la distance $d\in[0,w]$.
La longueur $h^{*}(d)$ est la somme des longueurs des arcs d'horocycles contenus dans les bandes.
L'ordre de ces horocycles est sans incidence sur cette longueur et l'on peut donc supposer que $h^{*}(d)$ est la longueur
d'une concatenation de deux arcs d'horocycles de convexit\'es invers\'ees contenus dans deux bandes 
d'\'epaisseurs $a_{p}=\sum_{k=1}^{N}a_{2k}$ et $a_{i}=\sum_{k=1}^{N}a_{2k-1}$.
On a
$$
h^{*}(d)=a_{p}\,e^{-d}+a_{i}\,e^{-w+d}.
$$
On recherche le minimum $h^{*}$ de cette function lorsque $d$ varie dans $[0,w]$.
En regardant o\`u la d\'eriv\'ee s'annule, on obtient que le minimum est unique, 
atteint pour $e^{d}=\sqrt{\frac{a_{p}}{a_{i}}}e^{w/2}$ et vaut

\begin{center}
\framebox[1,2\width][c]{
$h^{*}=2\sqrt{a_{p}\,a_{i}}\,e^{-w/2}$.}
\end{center}
$\ $\\

Notons $h$ la longueur de l'arc g\'eod\'esique joignant le segment de $\mu\cap C$
s\'eparant $B_{1}$ et $B_{2N}$ \`a lui-m\^eme perpendiculairement.
On appelle cette longueur une \textbf{hauteur} du cylindre $C$.\\

On se r\'ef\`ere \`a la figure \ref{bande2}.
Coupons le cylindre $C$ le long de la feuille s\'eparant $B_{1}$ et $B_{2N}$. 
On place le cylindre coup\'e dans le demi-plan sup\'erieur hyperbolique comme indiqu\'e, de telle sorte
qu'un des c\^ot\'es g\'eod\'esique soit contenu dans la verticale issue de $O$ \`a partir du point
d'ordonn\'ee 1.
Le nombre $h$ est la longueur de l'arc de cercle joignant perpendiculairement la verticale issue de $O$ 
et le demi-cercle passant par les points situ\'es sur l'axe des abscisses aux valeurs $x_{2N-1}$ et $x_{2N}$.
Le bord gauche de $C$, situ\'e en bas, est la concat\'enation des arcs d'horocycles de longueurs 
$a_{1},a_{2}e^{-w},a_{3},a_{4}e^{-w},\cdots,a_{2N-1},a_{2N}e^{-w}$.
Le bord droit de $C$, situ\'e en haut, est la concat\'enation des arcs d'horocycles de longueurs 
$a_{1}e^{-w},a_{2},a_{3}e^{-w},a_{4}\cdots,a_{2N-1}e^{-w},a_{2N}$.
La bande $B_{j}$, $j\geq 1$, est bord\'ee par les g\'eod\'esiques joignant l'abscisses $x_{j-1}$ aux abscisses $x_{j-2}$ et $x_{j}$
(avec $x_{0}=\infty$ et $x_{-1}=0$).
Remarquons tout d'abord que l'on a 
$$
x_{2j}=a_{1}+\cfrac{1}{a_{2}e^{-w}+
  \cfrac{1}{a_{3}+
    \cfrac{1}{\ddots
      \cfrac{1}{a_{2j-1}+
	\cfrac{1}{a_{2j}e^{-w}
}}}}}
$$
et

$$
x_{2j-1}=a_{1}+
\cfrac{1}{a_{2}e^{-w}+
  \cfrac{1}{a_{3}+
    \cfrac{1}{\ddots
      \cfrac{1}{a_{2j-2}e^{-w}+
	\cfrac{1}{a_{2j-1}
}}}}},
$$

ce que l'on note
$$
x_{2j}=[a_{1},a_{2}e^{-w},\cdots,a_{2j}e^{-w}]\ \textrm{et}\ x_{2j-1}=[a_{1},a_{2}e^{-w},\cdots,a_{2j-1}].
$$
En effet, l'isom\'etrie qui envoie la g\'eod\'esique verticale issue de $0$ sur la g\'eod\'esique joignant les
abscisses $x_{j-1}$ et $x_{j}$ est donn\'ee par la composition $P_{1}\circ P_{2}\circ\cdots\circ P_{j}(0)$
des isom\'etries de type parabolique
$$
P_{k}\ :\ z\mapsto
\left\{
  \begin{array}{ll}
    z+a_{k} & \textrm{si}\ k\ \textrm{est impair},\\
    \frac{1}{\bar{a}_{k}+\frac{1}{z}},& \textrm{si}\ k\ \textrm{est pair},\ \textrm{o\`u}\ \bar{a}_{k}=a_{k}e^{-w}.
  \end{array}
\right.
$$
Notez que l'ordre de composition est invers\'e.
Le point $x_{j}$ est l'image de z\'ero ou de l'infini par ce produit d'isom\'etries, suivant la parit\'e de $j$.
Plus pr\'ecis\'ement, pour $j\in\{1,\cdots,N\}$,
$$
x_{2j-1}=P_{1}\circ P_{2}\circ\cdots P_{2j-1}(0)\ \textrm{et}\ x_{2j}=P_{1}\circ P_{2}\circ\cdots P_{2j}(\infty).
$$
$\ $\\

Soit $\theta$ l'angle \`a l'origine entre la verticale issue de l'origine et la courbe \'equidistante 
tangente \`a la g\'eod\'esique joignant les abscisses $x_{2N-1}$ et $x_{2N}$.
Un peu de g\'eom\'etrie hyperbolique permet d'\'etablir
$$
\cosh(h)=\frac{1}{\cos(\theta)}.
$$
En effet, si l'on param\'etrise l'arc de cercle $x^{2}+y^{2}=1$ \`a l'aide de l'abscisse $x$, 
on obtient $h=\int_{0}^{\sin\theta}\frac{\sqrt{dy^{2}(x)+dx^{2}}}{y(x)}=\int_{0}^{\sin\theta}\frac{\sqrt{x^{2}/y(x)^{2}+1}}{y(x)}dx
=\int_{0}^{\sin\theta}\frac{dx}{1-x^{2}}$. 
Donc $\tanh(h)=\sin(\theta)$.
On en d\'eduit la formule.

\begin{figure}[!hbp]
\centering
\psfrag{i}{\small $i$}
\psfrag{0}{\small $0$}
\psfrag{h}{\small $h$}
\psfrag{t}{\small $\theta$}
\psfrag{x1}{\small $x_{1}$}
\psfrag{x2}{\small $x_{2}$}
\psfrag{x3}{\small $x_{3}$}
\psfrag{x4}{\small $x_{4}$}
\psfrag{x5}{\small $x_{5}$}
\psfrag{x8}{\small $x_{6}$}
\psfrag{a1}{\small $a_{1}$}
\psfrag{a2}{\small $a_{2}$}
\psfrag{a3}{\small $a_{3}$}
\psfrag{a5}{\small $a_{5}$}
\psfrag{a1e}{\small $a_{1}e^{-w}$}
\psfrag{a2e}{\small $a_{2}e^{-w}$}
\includegraphics[width=.9\linewidth]{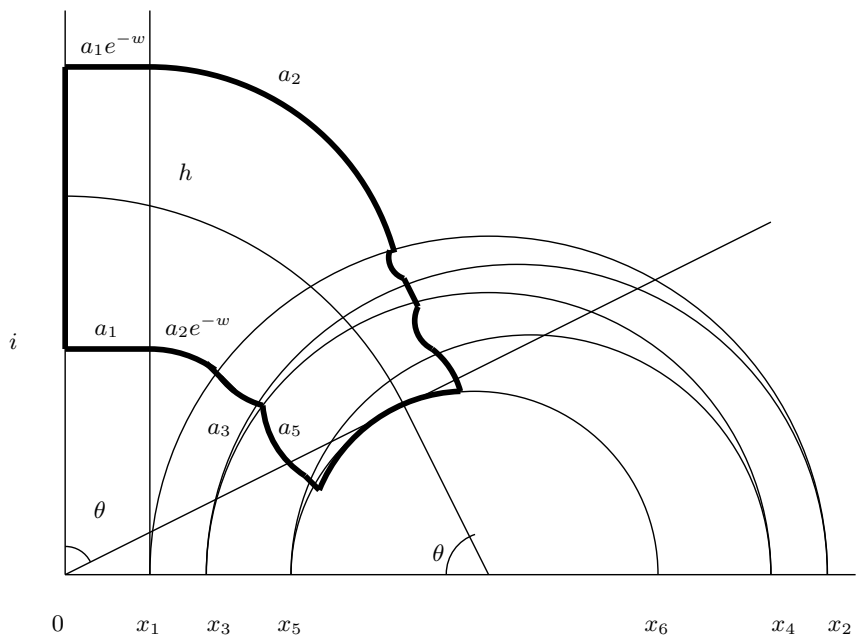}
\caption{\small{Ce dessin repr\'esente le cylindre $C$ (coup\'e le long d'une feuille de $\mu$) et sa 
d\'ecomposition en bandes $B_{1},\cdots,B_{2N}$, dans le demi-plan sup\'erieur hyperbolique.
Nous avons indiqu\'e les \'epaisseurs des bandes ainsi que les abscisses $x_{1},x_{2},\ldots$}}
\label{bande2}
\end{figure}

Revenons maintenant \`a la figure de d\'epart.
On a 
$$
\cos(\theta)=\frac{x_{2N}-x_{2N-1}}{x_{2N}+x_{2N-1}}.
$$
Finalement, on obtient

\begin{center}
\framebox[1,2\width][c]{$\cosh(h)=1+\frac{2}{\frac{x_{2N}}{x_{2N-1}}-1}$.}
\end{center}

\subsection{Estimation asymptotique le long d'une ligne \'etirement}

D\'eformons maintenant la structure hyperbolique $h$ en l'\'etirant le long de $\mu$.
Nous allons donner une estimation asymptotique de la longueur des composantes $\lambda_{j}$.
On reprend les notations du paragraphe pr\'ec\'edent mais on ajoute le param\`etre $t$.
Notons $\ell(t)$ la longueur du c\oe ur g\'eod\'esique du cylindre $C$. 
On a l'encadrement
$$
h(t)\leq\ell(t)\leq h^{*}(t).
$$

La longueur de tout arc d'horocycle contenu dans la pointe d'un triangle id\'eal est inf\'erieure ou \'egale \`a un.
Elle vaut un si et seulement si l'arc d'horocycle est un c\^ot\'e d'une r\'egion non-feuillet\'ee du feuilletage horocyclique.
Lorsqu'on effectue un \'etirement de temps $t\geq 0$, cette longueur est \'elev\'ee \`a la puissance $e^{t}$.
Ainsi, la longueur d'un tel arc d'horocycle converge vers z\'ero lorsque $t$ tend vers l'infini
si et seulement si elle est strictement plus petite que 1.
Sinon, elle reste constante \'egale \`a 1.
On en conclut que l'\'epaisseur d'une bande est une fonction de $t$ d\'ecroissante et qu'elle a pour limite,
lorsque $t$ tend vers l'infini, le nombre de r\'egions non-feuillet\'ees adjacentes \`a la bande. 

Nous allons maintenant donner un \'equivalent de $h(t)$ lorsque $t$ tend vers l'infini.
Notons d'abord que, par d\'efinition de la mesure transverse du feuilletage horocyclique, on a $w(t)=e^{t}w$.

\begin{lemma}
On a, lorsque $t$ tend vers l'infini, les \'equivalences suivantes
$$
x_{2N-1}(t)\sim a_{i}\ \textrm{et}\  x_{2N}(t)\sim\frac{1}{a_{p}}e^{e^{t}w}.
$$
\end{lemma}

\begin{proof}[D\'emonstration]
Notez que l'on a 
$$
x_{2j}(t)\sim [a_{1}(t),a_{2}(t)e^{-w(t)},\cdots,(a_{2j-2}(t)+a_{2j}(t))e^{-w(t)}]
$$
et 
$$
x_{2j-1}(t)\sim [a_{1}(t),a_{2}(t)e^{-w(t)},\cdots,a_{2j-3}(t)+a_{2j-1}(t)].
$$

On en d\'eduit, par une r\'ecurrence \'evidente, 
$$
x_{2j}(t)\sim [a_{1}(t),(a_{2}(t)+\cdots+a_{2j-2}(t)+a_{2j}(t))e^{-w(t)})]
$$
et
$$
x_{2j-1}(t)\sim a_{1}(t)+a_{3}(t)+a_{2j-3}(t)+a_{2j-1}(t).
$$
Par cons\'equent, on a 
$$
x_{2j}(t)\sim \Big{(}a_{2}(t)+\cdots+a_{2j-2}(t)+a_{2j}(t)\Big{)}^{-1}e^{w(t)}
$$
et
$$
x_{2j-1}(t)\sim a_{1}(t)+a_{3}(t)+a_{2j-3}(t)+a_{2j-1}(t).
$$
Notez que les sommes ci-dessus sont, pour $j=N$, non nulles et convergent 
vers les nombres entiers que l'on a not\'es $a_{p}$ et $a_{i}$ plus haut.
La d\'emonstration est achev\'ee.
\end{proof}

Du lemme pr\'ec\'edent et de la formule exprimant $\cosh(h)$, il vient

\begin{center}
\framebox[1,2\width][c]{$h(t)\sim 2\sqrt{a_{p}a_{i}}e^{-e^{t}w/2}$.}
\end{center}

De plus, on a facilement

\begin{center}
\framebox[1,2\width][c]{$h^{*}(t)\sim 2\sqrt{a_{p}a_{i}}e^{-e^{t}w/2}$.}
\end{center}

L'encadrement $h(t)\leq\ell(t)\leq h^{*}(t)$ permet de conclure 
\begin{center}
\framebox[1,2\width][c]{$\ell(t)\sim 2\sqrt{a_{p}a_{i}}e^{-e^{t}w/2}$.}
\end{center}

La proposition \ref{pr:equiv_elem} est d\'emontr\'ee.

\section{Divergence des lignes d'\'etirement \'el\'ementaires cylindriques}

Soient $t\mapsto g_{t}$ et $t\mapsto h_{t}$ deux lignes \'el\'ementaires de supports respectifs $\mu$ et $\nu$.
Supposons que leurs feuilletages horocycliques $F_{\mu}(g)$ et $F_{\nu}(h)$ correspondent \`a la m\^eme multicourbe 
$\lambda=\lambda_{1}\cup\cdots\cup\lambda_{M}$ avec des pond\'erations respectives $w_{j}(g)$, $w_{j}(h)$, 
$j\in\{1,\cdots,M\}$. 

Posons $\delta_{j}(g,h)=(w_{j}(h)-w_{j}(g))/2$.
D'apr\`es ce qui pr\'ec\`ede, on peut \'ecrire 
$$
\frac{\ell_{h_{t}}(\lambda_{j})}{\ell_{g_{t}}(\lambda_{j})}\sim K\,e^{-e^{t}\delta_{j}(g,h)},
$$
o\`u $K$ est une constante strictement positive qui ne d\'epend que des laminations $\mu$ et $\nu$. 

Par cons\'equent, 
comme le rapport $\frac{\ell_{h_{t}}(\lambda_{j})}{\ell_{g_{t}}(\lambda_{j})}$ minore l'exponentielle de la distance $d_{\mathcal{T}}(g_{t},h_{t})$,
on a
\begin{center}
\framebox[1,2\width][c]{
$\delta_{j}(g,h)<0\Longrightarrow \lim_{t\to\infty}d_{\mathcal{T}}(g_{t},h_{t})=\infty$.}
\end{center}

En particulier, on obtient le r\'esultat suivant

\begin{proposition}
La distance de Thurston n'est pas sym\'etrique le long d'une ligne d'\'etirement \'el\'ementaire cylindrique.
Plus pr\'ecis\'ement, si $t\mapsto h_{t}$ est un ligne d'\'etirement \'el\'ementaire cylindrique, on a,
pour tout $c>0$,
$$
\forall t,\ d_{\mathcal{T}}(h_{t},h_{t+c})=c\ \textrm{et}\ \lim_{t\to\infty}d_{\mathcal{T}}(h_{t+c},h_{t})=\infty.
$$
\end{proposition}

\begin{proof}[D\'emonstration]
L'\'egalit\'e $d_{\mathcal{T}}(h_{t},h_{t+c})=c$ vient du fait que la ligne d'\'etirement est une g\'eod\'esique.
Pour \'etablir la limite, notez que, pour tout $j\in\{1,\cdots,M\}$, 
$\delta_{j}(h_{c},h_{0})=(w_{j}(h)-w_{j}(h_{c}))/2=w_{j}(h)(1-e^{c})/2<0$.
Donc $\lim_{t\to\infty}d_{\mathcal{T}}(h_{t+c},h_{t})=\infty$.
\end{proof}

Nous pouvons maintenant d\'emontrer le r\'esultat qui suit.

\begin{proposition}
Deux lignes d'\'etirement \'el\'ementaires dont les directions sont des 
classes projectives distinctes de la m\^eme multicourbe divergent.
\end{proposition}

\begin{proof}[D\'emonstration]
Consid\'erons deux lignes d'\'etirement comme dans l'\'enonc\'e de la proposition.
Notons $\mu$ et $\nu$ leurs supports.
Fixons un point base $g$ sur la ligne d'\'etirement de support $\mu$ et 
un point base $h$ sur la ligne d'\'etirement de support $\nu$.
Par hypoth\`ese, les directions des deux lignes sont des classes projectives associ\'ees 
\`a la m\^eme multicourbe $\lambda=\lambda_{1}\cup\cdots\cup\lambda_{M}$.
Notons $(w_{1},\cdots,w_{M})$ et $(w'_{1},\cdots,w'_{M})$ les pond\'erations sur les composantes de $\lambda$
qui co\"incident avec les largeurs des cylindres des feuilletages horocycliques $F_{g}(\mu)$ et $F_{h}(\nu)$
respectivement.
Par hypoth\`ese, les vecteurs $(w_{1},\cdots,w_{M})$ et $(w'_{1},\cdots,w'_{M})$ ne sont pas colin\'eaires.

Fixons la param\'etrisation $t\mapsto g_{t}$ pour la ligne d'\'etirement passant par $g$ et de support $\mu$
de telle sorte que $g_{0}=g$.
Pour montrer que les lignes divergent, il suffit de trouver une param\'etrisation  $t\mapsto h_{t}$ de la ligne passant par $h$ et de support $\nu$
telle que les distances $d_{\mathcal{T}}(g_{t},h_{t})$ et $d_{\mathcal{T}}(h_{t},g_{t})$
tendent vers l'infini lorsque $t$ tend vers $+\infty$.
Cela revient \`a choisir le point base $h_{0}$ de la param\'etrisation.

D'apr\`es ce qui pr\'ec\`ede, il suffit de montrer qu'il existe un point base $h_{0}$ et deux indices $j_{0},j_{1}\in\{1,\cdots,M\}$ 
tels que $\delta_{j_{0}}(g,h_{0})=(w_{j_{0}}(h_{0})-w_{j_{0}})/2<0$ et $\delta_{j_{1}}(h_{0},g)=(w_{j_{1}}-w_{j_{1}}(h_{0}))/2<0$.
Pour tout $j\in\{1,\cdots,M\}$, il existe un unique $u\in\mathbb{R}$ tel que $w_{j}(h_{0})=e^{u}\,w_{j}'$.
Ainsi, nous sommes ramen\'es \`a prouver l'existence d'un nombre r\'eel $u$ et de deux indices $j_{0},j_{1}\in\{1,\cdots,M\}$ 
tels que $e^{u}\,w'_{j_{0}}<w_{j_{0}}$ et $e^{u}\,w'_{j_{1}}>w_{j_{1}}$.

On peut sans perte de g\'en\'eralit\'e supposer que $M=2$ et que, quitte \`a modifier les points base $g$ et $h$, on ait $w_{1}=w'_{1}=1$.
Posons $w_{2}=w$ et $w'_{2}=w'$.
La non colin\'earit\'e des vecteurs $(1,w)$ et $(1,w')$ signifie que le rapport $w'/w$ n'est pas \'egal \`a un.
Quitte \`a \'echanger les r\^oles des lignes d'\'etirement, on peut supposer que $w'/w>1$.
On cherche donc un nombre r\'eel $u$ tel que 
$$
\log(e^{u}/1)\log(e^{u}w'/w)<0,
$$
c'est-\`a-dire tel que
$$
u(u+\log(w'/w))<0.
$$
Il est facile de voir que n'importe quel nombre $u$ dans l'intervalle $]0,\log(w'/w)[$ convient.
La d\'emonstration est achev\'ee.
\end{proof}

\section{Divergence des lignes cylindriques}

Nous passons maintenant au cas g\'en\'eral, o\`u $\mu$ n'est plus n\'ecessairement \'el\'ementaire.

\begin{theorem}
\label{th:asymp}
Soit $\mu$ une lamination g\'eod\'esique compl\`ete r\'ecurrente par cha\^ines.
Consid\'erons une ligne d'\'etirement cylindrique, $t\mapsto h_{t}$, 
de support $\mu$ et de direction la classe projective d'une multicourbe $\lambda$ \'equip\'ee d'une pond\'eration fix\'ee.
Si $\lambda_{j}$ d\'esigne une composante de $\lambda$ et $w_{j}$ le poids associ\'e, on a l'estimation suivante, 
quand $t$ tend vers l'infini,
$$
\ell_{h_{t}}(\lambda_{j})\sim 2\sqrt{K_{j}}e^{-e^{t}w_{j}/2},
$$
o\`u $K_{j}\in\mathbb{N}$ est le nombre de r\'egions non-feuillet\'ees du feuilletage horocyclique adjacentes \`a un bord du cylindre 
de c\oe ur $\lambda_{j}$ multipli\'e par le nombre de r\'egions non-feuillet\'ees adjacentes \`a l'autre bord. 
\end{theorem}

Avant de commencer la d\'emonstration, faisons quelques remarques.
Soit $C$ un cylindre du feuilletage horocyclique, de largeur $w(t)=e^{t}w$.
Le cylindre $C$ est travers\'e en g\'en\'eral par un nombre ind\'enombrable de feuilles de $\mu$ et il
n'y a pas de notion canonique de d\'ecomposition en bandes dans ce cas.
Nous allons utiliser une approximation finie de $\mu$ et sa d\'ecomposition en bandes.

Avant cela, montrons qu'il existe une unique feuille du feuilletage horocyclique de $C$ dont la longueur est minimale.
Tout d'abord, la longueur d'une feuille $\alpha$ du feuilletage horocyclique est la somme des longueurs des arcs d'horocycles 
de $\alpha\setminus\mu$, car la lamination $\mu$ est de mesure de Lebesgue nulle.
Ensuite, le nombre de ces arcs d'horocycles est d\'enombrable et la longueur de $\alpha$ ne d\'epend pas de l'ordre dans lequel
ces arcs apparaissent.
Ainsi, quitte \`a r\'eorganiser les arcs d'horocycles, 
on peut supposer que le cylindre est compos\'e de deux bandes et que 
la feuille $\alpha$ est compos\'ee de deux arcs d'horocycles dont les convexit\'es 
sont de signes oppos\'es.
Nous avons vu pr\'ec\'edemment qu'il y a dans ce cas un unique minimum $h^{*}(t)$ valant
$$
h^{*}(t)=2\sqrt{a_{p}(t)\,a_{i}(t)}\,e^{-w(t)/2},
$$
o\`u $a_{p}(t)$ et $a_{i}(t)$ d\'esignent les \'epaisseurs des deux bandes.\\

On appelle \textbf{tron\c con} une bande minimale d'\'epaisseur 1, c'est-\`a-dire 
une bande minimale dont le bord contient le c\^ot\'e d'une r\'egion non-feuillet\'ee.

Consid\'erons l'approximation suivante du cylindre $C$.
Pour tout $t$, soit $C'$ le cylindre obtenu en ne gardant du cylindre $C$ que les
tron\c cons et en recollant les bords g\'eod\'esiques de ces tron\c cons par des isom\'etries.
Notez que la largeur de $C'$ est \'egalement $w(t)=e^{t}w$.
Notons $h'(t)$ la hauteur du cylindre $C'$ et $h(t)$ la hauteur du cylindre $C$.

\begin{proof}[D\'emonstration du th\'eor\`eme \ref{th:asymp}]
La d\'emonstration repose sur l'in\'egalit\'e suivante :
\begin{center}
\framebox[1,2\width][c]{
$\forall t,\ h'(t)\leq h(t)$.}
\end{center}

Cette in\'egalit\'e \`a \'et\'e \'etablie dans \cite{the2}, p. 398, Lemma 2.8.
Sa preuve consiste \`a remarquer que si l'on ins\`ere un tron\c con \`a $C'$ pour obtenir un cylindre $C''$, 
la nouvelle hauteur $h''(t)$ est strictement plus grande que celle de $C'$.
Un petit dessin dans le demi-plan sup\'erieur comme celui de \cite{the2} permet de s'en convaincre.

On a les encadrements
$$
h'(t)\leq h(t)\leq \ell(t)\leq h^{*}(t),
$$
o\`u $\ell(t)$ d\'esigne la longueur du c\oe ur du cylindre $C$ pour la m\'etrique $h_{t}$.

On a vu pr\'ec\'edemment que les termes $h'(t)$ et $h^{*}(t)$ sont tous deux \'equivalents, 
lorsque $t$ tend vers l'infini, \`a
$$
2\sqrt{K}\,e^{-w(t)/2},
$$
o\`u $K\in\mathbb{N}$ est le nombre de r\'egions non-feuillet\'ees adjacentes \`a un c\^ot\'e de $C$ multipli\'e
par le nombre de r\'egions non-feuillet\'ees adjacentes \`a l'autre c\^ot\'e de $C$.
La d\'emonstration est termin\'ee.
\end{proof}

On en d\'eduit imm\'ediatement, par les m\^emes d\'emonstrations, les versions g\'en\'erales des th\'eor\`emes \'etablis 
dans le cadre des lignes d'\'etirement \'el\'ementaires.

\begin{theorem}
\label{th:div}
Deux lignes d'\'etirement cylindriques dont les directions correspondent \`a des 
classes projectives distinctes de la m\^eme multicourbe divergent.
\end{theorem}

\begin{theorem}
Soit $t\mapsto h_{t}$ une ligne d'\'etirement cylindrique.
Pour tout nombre $c>0$, on a 
$$
\forall t,\ d_{\mathcal{T}}(h_{t},h_{t+c})=c\ \textrm{et}\ \lim_{t\to\infty}d_{\mathcal{T}}(h_{t+c},h_{t})=\infty.
$$
\end{theorem}

\begin{remark}
Consid\'erons plusieurs composantes $\lambda_{j},\ j\in J\subset\{1,\cdots,M\}$ de $\lambda$.
Posons $w_{min}:=\inf_{j\in J}\{w_{j}\}$ et $J_{min}:=\{j\in J\ :\ w_{j}=w_{min}\}$.
On obtient l'estimation suivante, quand $t$ tend vers l'infini,
$$
\ell_{h_{t}}(\cup_{j\in J}\lambda_{j})\sim 2\Big{(}\sum_{j\in J_{min}}\sqrt{K_{j}}\Big{)}e^{-e^{t}w_{min}/2}.
$$
\end{remark}

\section{Le probl\`eme des lignes parall\`eles et divergentes}

Nous venons de d\'emontrer que deux lignes cylindriques de directions diff\'erentes mais topologiquement \'egales 
divergent.
En fait, on a le r\'esultat g\'en\'eral suivant.

\begin{proposition}
Deux lignes d'\'etirement cylindriques sont parall\`eles seulement si elles ont la m\^eme direction.
\end{proposition}

\begin{proof}[D\'emonstration]
Montrons la contrapos\'ee, c'est-\`a-dire que deux lignes d'\'etirement cylindriques 
dont les directions sont diff\'erentes ne sont pas parall\`eles.
Notons $\lambda$ et $\lambda'$ les directions de deux lignes d'\'etirement cylindriques avec $\lambda\neq\lambda'$.
Les objets $\lambda$ et $\lambda'$ sont donc les classes projectives de multicourbes pond\'er\'ees. 
Supposons d'abord que les directions $\lambda$ et $\lambda'$ sont topologiquement distinctes.
Le th\'eor\`eme principal de \cite{the1} affirme  qu'on a, le long d'une ligne d'\'etirement $t\mapsto h_{t}$ de souche $\gamma$
et de direction $\lambda$, les comportements asymptotiques suivants.
Soit $\alpha$ une lamination g\'eod\'esique mesur\'ee. 
La longueur $\ell_{h_{t}}(\alpha)$ de $\alpha$
\begin{itemize}
\item converge vers $0$ si $\alpha\subset\gamma$,
\item converge vers l'infini, si $i(\alpha,\gamma)\neq 0$,
\item est born\'ee dans $\mathbb{R}_{+}^{*}$, si $\alpha\cap\gamma=\emptyset$.
\end{itemize}
Par cons\'equent, si $t\mapsto g_{t}$ et $t\mapsto h_{t}$ sont les lignes d'\'etirement de directions $\lambda$ et $\lambda'$, 
les rapports 
$$
\frac{\ell_{h_{t}}(\lambda)}{\ell_{g_{t}}(\lambda)}\ \textrm{et}\ \frac{\ell_{g_{t}}(\lambda')}{\ell_{h_{t}}(\lambda')}
$$
convergent vers l'infini, ce qui montrent que les lignes divergent.
Cette conclusion ne d\'epend pas de la param\'etrisation positive par la longueur d'arc choisie.
Les lignes ne sont donc pas parall\`eles.

Supposons donc que les directions $\lambda$ et $\lambda'$ soient topologiquement identiques.
Le th\'eor\`eme \ref{th:div} permet de conclure que les lignes divergent alors.
La d\'emonstration est achev\'ee.
\end{proof}

Nous allons maintenant \'etablir la r\'eciproque.

\begin{theorem}
Deux lignes d'\'etirement cylindriques sont parall\`eles si et seulement si elles ont la m\^eme direction.
\end{theorem}

\begin{proof}[D\'emonstration]
Il reste \`a montrer que deux lignes d'\'etirement cylindriques de m\^eme directions sont parall\`eles.
Consid\'erons une g\'eod\'esique simple ferm\'ee $\alpha$ et une ligne d'\'etirement $t\mapsto h_{t}$
de direction $\lambda$.
Notons $C_{1},\cdots,C_{M}$ les cylindres associ\'es aux composantes $\lambda_{1},\cdots,\lambda_{M}$ de $\lambda$.
Fixons \'egalement dans chaque cylindre un arcs g\'eod\'esique $\rho_{j}$ de longueur $w_{j}(t)$, $j=1,\cdots,M$.
La courbe $\alpha$ traverse les cylindres et effectue un certains nombres de tours autour des composantes de $\lambda$.
Nous avons donc l'encadrement suivant
$$
\sum_{j=1}^{M}|\alpha\cap\lambda_{j}|w_{j}(t)\leq \ell_{h_{t}}(\alpha)\leq
\sum_{j=1}^{M}\Big{(}|\alpha\cap\lambda_{j}|w_{j}(t)+|\alpha\cap\rho_{j}|\ell_{h_{t}}(\lambda_{j})\Big{)}.
$$
Cet encadrement montre que si $i(\alpha,\lambda)\neq 0$, la longueur $\ell_{h_{t}}(\alpha)$ est de l'ordre de $e^{t}$.

On en d\'eduit que, si  $t\mapsto g_{t}$ est une ligne d'\'etirement de direction $\lambda$, 
les rapports $\frac{\ell_{h_{t}}(\alpha)}{\ell_{g_{t}}(\alpha)}$ et $\frac{\ell_{g_{t}}(\alpha)}{\ell_{h_{t}}(\alpha)}$,
avec $i(\alpha,\lambda)\neq 0$, sont born\'es.
Dans le cas o\`u $\alpha\cap\lambda=\emptyset$, le th\'eor\`eme de l'article \cite{the1} d\'ej\`a cit\'e dans la d\'emonstration pr\'ec\'edente
implique \'egalement que les rapports $\frac{\ell_{h_{t}}(\alpha)}{\ell_{g_{t}}(\alpha)}$ et $\frac{\ell_{g_{t}}(\alpha)}{\ell_{h_{t}}(\alpha)}$ sont born\'es.
Finalement, on a montr\'e plus haut dans le th\'eor\`eme \ref{th:asymp} que si $\alpha\subset\lambda$, les rapports sont eux aussi born\'es.

Nous venons donc d'\'etablir que pour toute multicourbe $\alpha$, il existe un nombre $M(\alpha)>0$ tel que, pour tout $t\geq0$, 
$$
\frac{1}{M(\alpha)}\leq\frac{\ell_{h_{t}}(\alpha)}{\ell_{g_{t}}(\alpha)}\leq M(\alpha).
$$

On conclut la d\'emonstration avec la proposition qui suit, en se souvenant qu'une ligne d'\'etirement converge vers sa direction \cite{pap}.
\end{proof}

\begin{proposition}
Soient $(g_{n})$ et $(h_{n})$ deux suites de l'espace de Teichm\"uller convergeant vers le m\^eme point de
la compactification $\mathcal{T}(\Sigma)\cup\mathcal{PL}(\Sigma)$ de l'espace de Teichm\"uller et
telles que, pour toute courbe simple ferm\'ee $\alpha$,
il existe un nombre $M(\alpha)>0$ tel que, pour tout $n\in\mathbb{N}$, 
$$
\frac{1}{M(\alpha)}\leq\frac{\ell_{h_{n}}(\alpha)}{\ell_{g_{n}}(\alpha)}\leq M(\alpha).
$$
Alors les suites $(d_{\mathcal{T}}(g_{n},h_{n}))$ et $(d_{\mathcal{T}}(h_{n},g_{n}))$ sont born\'ees.
\end{proposition}

\begin{proof}[D\'emonstration]
On raisonne par l'absurde.
Supposons que l'une des suites, disons $(d_{\mathcal{T}}(g_{n},h_{n}))$, ne soit pas born\'ee.
Quitte \`a extraire une sous-suite, on peut supposer que  $(d_{\mathcal{T}}(g_{n},h_{n}))$ converge vers l'infini.
Cela implique que pour tout $n$, il existe une g\'eod\'esique simple ferm\'ee $\alpha_{n}$ telle que
$$
\frac{\ell_{h_{n}}(\alpha_{n})}{\ell_{g_{n}}(\alpha_{n})}\geq n.
$$
Quitte \`a extraire une sous-suite, on peut supposer que la suite des classes projectives $([\alpha_{n}])$ converge
dans $\mathcal{PL}(\Sigma)$ vers la classe projective $[\alpha_{\infty}]$.
De plus, comme les suites $(g_{n})$ et $(h_{n})$ convergent, par hypoth\`ese, vers la m\^eme limite dans $\mathcal{PL}(\Sigma)$,
il existe deux suites $(x_{n})$ et $(y_{n})$ de r\'eels positifs telles que
$$
\forall \beta\in\mathcal{ML}(\Sigma),\ \lim_{n\to\infty}\frac{y_{n}\ell_{h_{n}}(\beta)}{x_{n}\ell_{g_{n}}(\beta)}=1.
$$ 
Par hypoth\`ese, on a, pour $\alpha$ donn\'e,
$$
\frac{y_{n}}{x_{n}}\frac{1}{M(\alpha)}\leq\frac{y_{n}\ell_{h_{n}}(\alpha)}{x_{n}\ell_{g_{n}}(\alpha)}\leq \frac{y_{n}}{x_{n}}M(\alpha).
$$ 
On en d\'eduit que la suite $(\frac{y_{n}}{x_{n}})$ est born\'ee dans $\mathbb{R}_{+}^{*}$.
Quitte \`a extraire une sous suite, on peut supposer que cette suite converge vers un nombre $c\in\mathbb{R}_{+}^{*}$.

Par continuit\'e de la fonctionnelle d'intersection sur l'espace des courants g\'eod\'esiques \cite{bon}, on a
$$
\lim_{n\to\infty}\frac{y_{n}\ell_{h_{n}}(\alpha_{n})}{x_{n}\ell_{g_{n}}(\alpha_{n})}=1,\ \textrm{soit}\
\lim_{n\to\infty}\frac{\ell_{h_{n}}(\alpha_{n})}{\ell_{g_{n}}(\alpha_{n})}=\frac{1}{c}.
$$
Nous obtenons une contradiction, ce qui montre que les suites 
$(d_{\mathcal{T}}(g_{n},h_{n}))$ et $(d_{\mathcal{T}}(h_{n},g_{n}))$ sont born\'ees.
\end{proof}


\begin{thebibliography}{99}

\bibitem{bon} F. Bonahon,
  \textsl{The geometry of Teichm\"uller space via geodesic currents}, 
  Invent. Math.  92  (1988), 139--162. 

\bibitem{masur} H. Masur,
  \textsl{On a class of geodesics in Teichm\"uller space},
  Annals of Math., Vol.102, No.2 (1975), 205--221. 

\bibitem{pap} A. Papadopoulos,
 \textsl{On Thurston's boundary of Teichm\"uller space and the
   extension of earthquakes},
 Topology and its Applications \textbf{41}, 147--177 (1991)

\bibitem{the1} G. Th\'eret,
  \textsl{On Thurston's Stretch Lines in Teichm\"uller Space},
preprint.

\bibitem{the2} G. Th\'eret,
 \textsl{On the negative convergence of Thurston's stretch lines
   towards the boundary of Teichm\"uller space},
 Annales Acad. Scien. Fennicae Math., {\bf 32} (2007), 381--408.

\bibitem{thu} W.P. Thurston,
 \textsl{Minimal Stretch Maps Between Hyperbolic Surfaces},
 arXiv:math/9801039.

\end{thebibliography}
\end{document}